\documentclass[11pt]{article}

\usepackage[a4paper, left=3cm, top=3cm, right=3cm, bottom=3cm ]{geometry}
\usepackage{authblk}
\usepackage[onehalfspacing]{setspace}

\usepackage[T1]{fontenc}

\usepackage[bookmarks=closed]{hyperref}
\usepackage{breakurl}

\usepackage{amsmath,amsfonts}

\usepackage{algorithmic}
\usepackage{algorithm}
\usepackage{array}
\usepackage[caption=false]{subfig}
\usepackage{textcomp}
\usepackage{stfloats}
\usepackage{hyperref}
\usepackage{url}
\usepackage{verbatim}
\usepackage{siunitx}
\usepackage{graphicx}
\usepackage{amssymb, amsthm}
\usepackage{cite}

\usepackage{enumitem}
\usepackage{multirow}

\newtheorem{thm}{Theorem}
\newtheorem{problem}[thm]{Problem}
\newtheorem{remark}[thm]{Remark}

\newtheorem{theorem}[thm]{Theorem}
\newtheorem{definition}[thm]{Definiton}

\newcommand{\R}{\mathbb R}

\newcommand{\sph}{\mathbb S}

\newcommand\norm[1]{\left\Vert#1\right\Vert}

\newcommand\inner[1]{\langle#1\rangle}
\newcommand\minner[1]{\bigl\langle#1\bigr\rangle}
\newcommand\binner[1]{\left\langle#1\right\rangle}

\newcommand{\Fd}{\mathbf{F}}
\newcommand{\Ad}{\mathbf{A}}
\newcommand{\Hd}{\mathbf{H}}
\newcommand{\Gd}{\mathbf{G}}

\newcommand{\So}{{\mathbf S}}
\newcommand{\Mo}{{\mathbf M}}
\newcommand{\Uo}{{\mathbf U}}

\newcommand{\stepsize}{\omega}

\newcommand{\signal}{X}
\newcommand{\virt}{Q}

\newcommand{\signalaux}{\xi}
\newcommand{\dataaux}{\eta}

\newcommand{\noise}{\delta}
\newcommand{\data}{Y}

\newcommand{\density}{\operatorname{f}}

\newcommand{\radon}{\mathcal{R}}
\newcommand{\Ho}{\mathbf H}

\newcommand{\Qo}{\mathbf Q}

\newcommand{\lsq}{\mathcal{D}}
\newcommand{\mle}{\mathcal{L}}

\newcommand{\rr}{x}
\newcommand{\ee}{e}

\newcommand{\Nx}{N_x}
\newcommand{\Ny}{N_y}
\newcommand{\Nc}{M}
\newcommand{\Nd}{B}
\newcommand{\Ne}{E}

\numberwithin{equation}{section}
\numberwithin{thm}{section}
\numberwithin{figure}{section}

\allowdisplaybreaks

\title{Derivative-Free iterative One-Step Reconstruction for Multispectral CT}

\date{\today}

\author{Thomas Prohaszka}

\affil{Institute of Basic Sciences in Engineering Science, University of Innsbruck\authorcr
Technikerstrasse 13, 6020 Innsbruck, Austria\authorcr
E-mail:  \texttt{Thomas.Prohaszka@student.uibk.ac.at}
 }
 
\author{Lukas Neumann}

\affil{Institute of Basic Sciences in Engineering Science, University of Innsbruck\authorcr
Technikerstrasse 13, 6020 Innsbruck, Austria\authorcr
E-mail:  \texttt{lukas.neumann@uibk.ac.at}
 }

\author{Markus Haltmeier}

\affil{Department of Mathematics, University of Innsbruck\authorcr
Technikerstrasse 13, 6020 Innsbruck, Austria\authorcr
E-mail:  \texttt{markus.haltmeier@uibk.ac.at}
 }

\begin{document}

\maketitle

\begin{abstract} 
Image reconstruction in Multispectral Computed Tomography (MSCT) requires solving a challenging nonlinear inverse problem, commonly tackled via iterative optimization algorithms. Existing methods necessitate computing the derivative of the forward map and potentially its regularized inverse. In this work, we present a simple yet highly effective algorithm for MSCT image reconstruction, utilizing iterative update mechanisms that leverage the full forward model in the forward step and a derivative-free adjoint problem. Our approach demonstrates both fast convergence and superior performance compared to existing algorithms, making it an interesting candidate for future work. We also discuss further generalizations of our method and its combination with additional regularization and other data discrepancy terms.

\medskip\noindent \textbf{Keywords:}  
Inverse problems; coupled physics problems; multispectral CT, derivative-free iteratzions, inverse problems.    
\end{abstract}

\section{Introduction}

Classical computed tomography (CT) is based on the inversion of the linear Radon transform, where a scalar-valued attenuation map $\mu \colon \mathcal{X} \to \R $ of the patient is recovered from observation of its Radon transform $\radon \mu \colon \mathcal{L} \to \R$ derived from projection data. Here and below $\mathcal{X} \subseteq \R^d$ is the image domain in $d=2, 3$ dimensions, and $\mathcal{L}$ is a set of integration lines. While sufficient in many applications, the linear problem ignores the polychromatic nature of the X-rays and the energy-dependent absorption characteristics of real-world objects. The sample is more accurately represented by a family of attenuation maps $\mu(\ee) \colon \mathcal{X} \to \R$ dependent on the photon energy $\ee \in (0, \infty)$. Recovering a single $\mu$ from projection data using a single energy bin results in a mixture of density maps from different energies resulting in severe non-uniqueness. Additionally, the nonlinearity results in severe beam hardening artifacts that may be partially accounted for by iterative algorithms or analytic modeling \cite{mcdavid1975spectral,kiss2023beam,pan2008anniversary, herman1979correction,van2011iterative,rigaud2017analytical}. In order to overcome such weaknesses, the idea of multispectral CT (MSCT) is to measure projection data for different energy bands, which are then used to reconstruct multiple attenuation maps. The reconstruction problem, however, becomes nonlinear and much more challenging than pure Radon inversion \cite{kazantsev2018joint,rigie2015joint,hu2019nonlinear,arridge2021overview,Mory_2018,heismann2009quantitative,maass2009image}. In this work, we develop a simple and efficient strategy for tackling the nonlinear  MSCT image reconstruction problem.

\subsection{Multispectral CT}

Specifically,  in this work we use the material decomposition paradigm in MSCT. In the material decomposition approach  it is assumes that the energy dependent attenuation maps $\mu(\ee)$ can  be written as $\mu(\ee, \rr) = \sum_{m=1}^M  \mu_m(\ee)  \density_m(\rr)$ where   $\density_1, \density_2, \dots ,  \density_M \colon \mathcal{X}  \to \R$ are the densities  of $M$ separate materials to be recovered and $ \mu_m(\ee)$ are known and tabled absorption characteristic  of the $m$-th material.  By collecting  projection data for several energy bands  the aim  is to recover the  material densities.  This  does  not  only allow to improve image quality but also offers  abroad range  of the applications as it reconstructs multiple images encoding different characteristics  of specific regions enable a deeper understanding of the  objects  under examination.  The recent significant advancement  in the manufacturing of energy-sensitive sensors  \cite{willemink2018photon,kreisler2022photon} has considerably increased the interest in MSCT.

Assuming  $B$ spectral measurements $\data_1, \dots , \data_\Nd$, the   material  decomposition  problem  in MSCS can be written as the problem of recovering  $\density_1, \dots, \density_M$  from data  
 \begin{equation}\label{eq:IP}
	\data_\Nd \simeq    \int_0^\infty  s_b( \ee)  \exp  \bigl[-  \radon  \bigl( \sum_{m=1}^M  \mu_m(\ee)  \density_m(\rr)    \bigr)  \bigr] \; \mathrm{d} \ee \quad \text{ for } b = 1, 2, \dots,  B  \,.
\end{equation}
Here $\radon (\density_m)$  is the Radon transform applied to the $m$-th material density map $\density_m$,  $ \exp$ applies the exponential function pointwise and $s_b(\ee)$ represents the  energy profile (effective spectrum) of the  $b$-th measurement.  Classical CT  would  correspond to the  unrealistic case where $s_b$ is a Dirac delta  function and  where applying the pointwise logarithm to \eqref{eq:IP}  results  in a linear  inverse problem. In MSCT one accounts for the fact that  $s(b)$ has finite energy covering and thus one has to work with the  full nonlinear problem~\eqref{eq:IP} for recovering the unknown density maps $\density_m$.

\subsection{Two-step and one-step algorithms}

Various algorithms have been developed for solving the nonlinear inverse problem~\eqref{eq:IP}. They can be broadly classified into two categories: two-step methods and one-step algorithms. The idea of earlier two-step methods is to perform Radon inversion and material decomposition in two separate steps. Material decomposition can be performed either in the projection domain $\mathcal{L}$ (before Radon inversion) or in the image domain $\mathcal{X}$ (after Radon inversion). Both methods have their specific advantages and disadvantages. The image-domain decomposition approach allows incorporating prior information about the objects that is naturally contained in the image domain $\mathcal{X}$. However, the nonlinear nature of the problem leads to approximate linear models that introduce severe reconstruction artifacts. The image-domain decomposition approach, on the other hand, allows working with the correct nonlinear model. However, the prior structure in the Radon domain is not directly available. See the works \cite{heismann2009quantitative,maass2009image,schmidt2009optimal,niu2014iterative,schirra2013statistical} and references therein for various proposed two-step approaches.

One-step methods reconstruct the material densities $\density_1, \dots, \density_M$ through iterative minimization techniques for solving  \eqref{eq:IP} and thus overcome the drawbacks of both two-step methods. For some one-step algorithms, we refer to \cite{Mory_2018, cai2013full, long2014multi, mechlem2017joint, weidinger2016polychromatic, barber2016algorithm,chen2021non,kazantsev2018joint,rigie2015joint}. Despite their superior performance, such one-step iterative algorithms are computationally expensive. Existing methods require many iterative steps due to poor conditioning of the problem or come with computationally expensive iterative steps.   The algorithms proposed in this paper are specific one-step algorithms that address these two drawbacks of existing one-step methods.
                
\subsection{Our contributions}

As our main contribution, we present a novel derivative-free algorithm designed to combine the advantages of one-step and two-step approaches. To achieve this, we introduce a simple and computationally efficient iterative update that incorporates appropriate preconditioning. Image reconstruction is performed in the image domain, which naturally allows for the inclusion of an image smoothness prior. It also integrates benefits of two-step approaches by separating iterative updates into two parts. Moreover, the main ingredient that makes the algorithm efficient is the use of the full nonlinear forward model for the direct problem but linearisation around zero for the adjoint problem. While avoiding computation and evaluation of the derivative of the forward map, this also allows including a simple channel preconditioning. Our method can be combined with additional regularization. However, in order to show the method in its pure form, we will not include such a modification.

\section{Mathematical modelling of  MSCT}

We assume that the object to be imaged lies in some domain $\mathcal{X} \subseteq \R^2$ and consists of a combination of $M$ different materials with densities $\density_m \colon \mathcal{X} \to \R$ with $m = 1, \dots, M$.
Each material has a separate mass attenuation coefficient $\mu_m \colon [0, \infty) \to [0, \infty) $ which is a known function of the X-ray energy $\ee$.   The total energy dependent (linear) X-ray attenuation coefficient is then given by
\begin{equation}\label{eq:mu}
 	\mu(\rr ,\ee) = 
	\sum_{m=1}^{M} \mu_m(\ee) \density_m(\rr) \quad \text{ for } (\rr, \ee) \in \mathcal{X} \times \R \,.  
\end{equation}
Assuming that the material specific attenuation functions $\mu_m(\ee)$ are known, the goal is to recover densities $\density_m$ from indirect X-ray measurements using different energy bins, which we describe next.

\subsection{Continuous model}

We start with continuous modeling, where the quantities involved are functions on continuous domains that will be discretized later. Suppose that X-ray energy with a known incident spectral density $I_0(\ee)$ is sent along a line $\ell \in \mathcal{L}$ from the source position to the detector position.   While propagating along $\ell$, the X-rays are attenuated according to $\mu (\rr, \ee)$ defined in \eqref{eq:mu}. This results in an outgoing spectral density $I_1(\ee) = I_0(\ee) \exp ( - \int_{\ell} \mu(\rr,\ee) d\ell(\rr) )$ at the detector.   The energy sensitive detector with the spectral profile $\hat s (\ee)$ records the integral $\int_0^\infty \hat s (\ee) I_1(\ee) d\ee$.     Denoting the product of the incident spectral density of the source and detector sensitivity by $s(\ee) \triangleq I_0 (\ee) \hat s(\ee)$ , referred to as the effective spectrum, the recorded data is given by          
\begin{align} \nonumber
\data(s, \ell) 
& = \int_0^\infty s(\ee) \exp \left( - \int_{\ell} \mu(\rr,\ee) d\ell (\rr) ) \right) d \ee  
 \\ \label{eq:mess1}
 &= \int_\R s(\ee) \exp \left( - \sum_{m=1}^{M} \mu_m(\ee) \int_{\ell} \density_m d\ell (\rr) \right) d\ee
\,. 
\end{align}
The data in equation \eqref{eq:mess1} represent a single measurement in MSCT.  The goal of material decomposition in MSCT is to determine the density distributions $\density_m$ from multiple multispectral X-ray measurements by varying the line $\ell$ and the effective spectra $s$.

For simplicity  of presentation we consider the parallel beam mode where any line $\ell = \ell(\theta, r)$ is parametrized  by its normal vector $\theta$ and its distance $r$ from the origin.    In this case $\int_{\ell}  \density_m  d\ell (\rr)  =  \radon \density_m (\theta, r)$ is given by the classical Radon transform  of $\density_m$. Assuming further  a total number of $B$ different  effective spectra and writing $\density = (\density_1, \dots, \density_M)$ we obtain the  continuous   MSCT forward model    
\begin{equation}\label{eq:multispektrVO}
	\data   = \biggl[ \int_b s_b(\ee) \exp \Bigl( -   \sum_{m=1}^{M} \mu_m(\ee)  \radon ( \density_m )  \Bigr) d\ee \biggr]_{b =1, \dots B}  \,.
\end{equation}
Equation~\eqref{eq:multispektrVO} gives the complete continuous forward model in material decomposition in MSCT.  The unknown $\density$ consists of $M$ functions $\density_1, \dots, \density_M$ defined on the image domain $\mathcal{X}$ and the data of $B$ functions $\data_1, \dots, \data_\Nd$ defined  on the projection domain $\mathcal{L} = \sph^1 \times \R$.  The methods that we describe, however, would also work with a three-dimensional image domain $\mathcal{X}$ and a general projection domain $\mathcal{L}$  of lines in $\R^3$.

\begin{figure}
\includegraphics[width=0.9\textwidth]{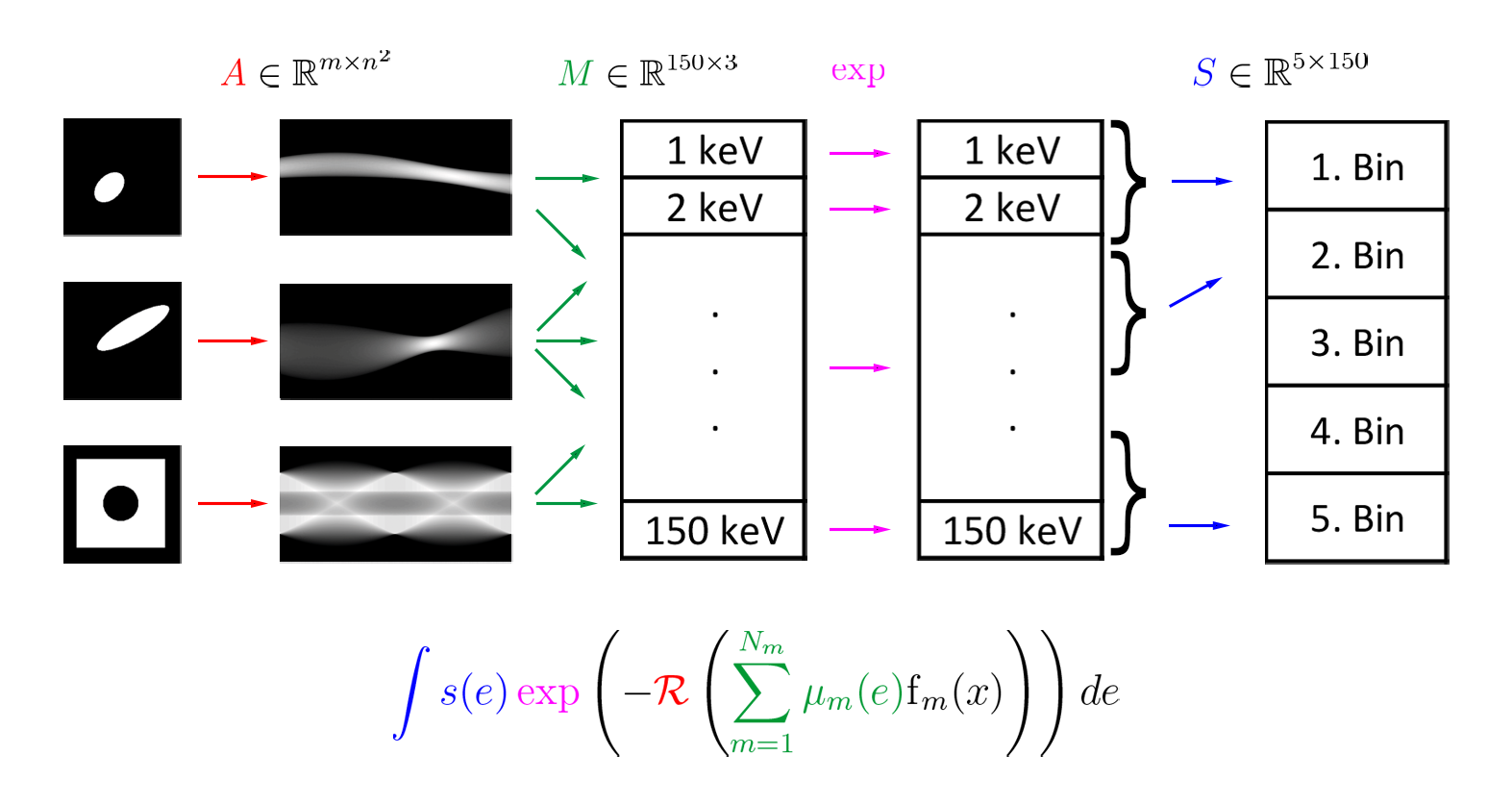}
\caption{ Illustration of the forward model in MSCT for $m=3$ materials and $b=5$ energy bins, using $E=150$ energy discretizations: First, the Radon transform is applied separately to each of the given material densities $\signal_1$, $\signal_2$, and $\signal_3$, resulting in three material sinograms, which can be seen as a three-channel sinogram. Next, the matrix $\Mo$ is applied to each pixel, resulting in 150 energy sinograms. To each of these sinograms, $x \mapsto \exp(-x)$ is applied, resulting in 150 virtual energy data maps. By applying the matrix $\So$ pixel by pixel, one obtains the final data consisting of data maps. The continuous forward model can be visualized in a similar way by replacing the material images with continuous counterparts and the 150 energy channels with a function-valued channel.}
\label{fig:forward}
\end{figure}

\subsection{Discretisation}

In order to avoid  technical  details and to concentrate  on the main  ideas we derive the algorithm for the   discrete forward model   throughout this paper.  For that  purpose we represent the material densities  via   discrete column vectors $ \signal_1, \dots ,  \signal_{\Nc}   \in \R^{\Nx}$ and the Radon transform via a matrix  $\Ad \in \R^{\Ny \times \Nx}$ where  $\Ny$ is the  number of lines used  in the projection domain. Further we discretise the effective energy spectra by vectors $\So_1, \dots, \So_B  \in \R^{\Ne}$ and the known material  attenuations by  vectors $\mu_1, \dots, \mu_{\Nc} \in \R^{\Ne}$. The discretization of \eqref{eq:multispektrVO} yields the following discrete image reconstruction problem.

\begin{problem}[Discrete  MSCT image reconstruction problem]
Recover the unknown  $\signal \in \R^{\Nx \times \Nc}$ from data  $\data = \Fd( \signal ) +  \noise \in \R^{\Ny \times  \Nd}$ where
\begin{equation} \label{eq:fwd}
  	\Fd( \signal )  \triangleq  (\So \cdot \boldsymbol{\exp}_{\Ne \times \Ny} (- \Mo  \cdot  ( \Ad \cdot  \signal)^\intercal ) )^\intercal  
  \end{equation}
Here and below we use the convention that the boldface notation $\boldsymbol{\exp}_{\Ne \times \Ny}$ indicates that the scalar function $\exp$ is applied pointwise to a vector  in $\R^{\Ne \times \Ny}$.   Further,  in \eqref{eq:fwd} , 
\begin{itemize}
\item the columns of $\signal = [\signal_1, \dots, \signal_M]$  are the discrete material images; 
\item $\Ad \in  \R^{\Ny \times \Nx}$ is the discretized Radon transform; 
\item the columns  of $\Mo  = [\Mo_1, \dots,  \Mo_{\Nc}] \in \R^{\Ne \times \Nc}$ are the discretized material attenuations;
\item the columns  of $\So^\intercal  = [\So_1^\intercal, \dots,  \So_\Nd^\intercal]  \in  \R^{\Ne \times \Nd}$  the discretized effective spectra;
\item the columns of $\data = [\data_1, \dots, \data_\Nd]$  are the observed spectral data.
\end{itemize}     
\end{problem}

Note  we repeatedly nclude the transpose operation $(\cdot)^\intercal$ in  \eqref{eq:fwd} such that all involved linear operations  can   be written as matrix multiplications  from  left. Alternatively we can also write $\Fd( \signal )    =  \boldsymbol{\exp}_{\Ny \times \Ne} (-   \Ad  \signal  \Mo^\intercal ) \So^\intercal$ reflecting  that the discrete Radon transform $\Ad$  operates on the columns and the matrices  $\Mo$ and  $\So$ on the rows of $\signal \in \R^{\Ny \times \Nd}$.  At some places we will denote operation  of $\Mo$  on  $\signal$  from the right  by  $\Mo \diamond \signal \triangleq \signal  \cdot   \Mo^\intercal$ such that we have $\Fd( \signal )    =  \So \diamond \boldsymbol{\exp}_{\Ny \times \Ne} (-   \Mo \diamond \Ad  \cdot \signal  )$ .  The structure  of the discrete forward  model   is  illustrated  in  Figure~\ref{fig:forward}.

\begin{remark}[Recalibration]
With  $ \inner{\So_b, \boldsymbol{1}}  = \sum_{e=1}^{\Ne} \So_{b,e}$ we obtain 
\begin{equation}  \label{eq:cal1}
	\Fd(0)^\intercal 
	= 
	\So  \cdot \boldsymbol{\exp}_{\Ne \times \Ny} ( 0 ) 
	=  
	\So  \cdot  \boldsymbol{1}_{\Ne \times \Ny} 
	= 
	\begin{pmatrix}
	 \inner{\boldsymbol{1}, \So_1}  & \cdots &  \inner{\boldsymbol{1}, \So_1} 
	 \\
	 \vdots  & & \vdots
	 \\
	 \inner{\boldsymbol{1}, \So_\Nd}   & \cdots  & \inner{\boldsymbol{1}, \So_\Nd}  
	\end{pmatrix} \,.
\end{equation}
Thus, with $\odot$  denoting  pointwise multiplication (also known as Hadamard product) and $\oslash $ the pointwise division we get 
\begin{multline} \label{eq:cal2}
	 \Fd(\signal)^\intercal \oslash \Fd(0)^\intercal
	= 
		\begin{pmatrix}
	 1/\inner{\boldsymbol{1}, \So_1}  & \cdots &  1/\inner{\boldsymbol{1}, \So_1} 
	 \\
	 \vdots  & & \vdots
	 \\
	 1/\inner{\So_b, \boldsymbol{1}}   & \cdots  & 1/\inner{\So_b, \boldsymbol{1}}  
	\end{pmatrix}
	 \odot ( \So \cdot \boldsymbol{\exp}_{\Ne \times \Ny} (- \Mo  \cdot  ( \Ad  \signal)^\intercal ) )
 	\\ = 
	\begin{pmatrix}
	\So_1 / \inner{\boldsymbol{1}, \So_1}   
	\\
	 \vdots 
	 \\
	\So_\Nd   / \inner{\boldsymbol{1}, \So_\Nd}   
	\end{pmatrix} \cdot \boldsymbol{\exp}_{\Ne \times \Ny} (- \Mo  \cdot  ( \Ad  \signal)^\intercal ) \,.
\end{multline}
This  means  that  the recalibrated forward model $ \Fd(\signal)^\intercal \oslash \Fd(0)^\intercal$ is the same as the original forward  model with  normalized  effective spectra $\So_b / \inner{\So_b, \boldsymbol{1}} $.   The matrix  with normalized spectra  can be written as  $\So  \oslash ( \So  \cdot  \boldsymbol{1}_{\Ne \times \Ne}  )$.    
 \end{remark}

In  this work we   use  rescaled data  to which we apply the pointwise logarithm $ \boldsymbol{\log}_{\Ny \times  \Nd}$ and the corresponding  least squares (LSQ) functional.  

\begin{definition}[Forward model and LSQ   functional]
We \label{def:operators} define  the logarithmic forward model  and the LSQ  functional   by    
\begin{align}  \label{eq:fwd-log}
  	&\Hd \colon \R^{\Nx \times \Nc} \to  \R^{\Ny \times  \Nd} 
	\colon \data \mapsto   \boldsymbol{\log}_{\Ny \times  \Nd}( \Fd(\signal) \oslash \Fd(0) )
	\\ \label{eq:lsq}
	&\lsq  \colon  \R^{\Nx \times \Nc} \to \R \colon 
	\signal \mapsto   \norm{\Hd( \signal) -  \data_{\Hd} }_2^2 / 2 \,.
\end{align}
Here  $\Fd( \signal )    =  \boldsymbol{\exp}_{\Ny \times \Ne} (-   \Ad  \signal  \Mo^\intercal ) \So^\intercal$ is  defined  in  \eqref{eq:fwd}, $\data \in   \R^{\Ny \times  \Nd} $ are the given  data, $\data_\Hd = \boldsymbol{\log}_{\Ny \times  \Nd}( \data \oslash \Fd(0)) $ the modified data, and $\boldsymbol{\log}_{\Ny \times  \Nd}(\cdot)$  the point-wise  logarithm.
\end{definition}

Using the  notations of Definition \ref{def:operators},  material decomposition in MSCT amounts to the near solution of  $\Hd( \signal) = \data_{\Hd}  $ or the near  minimization of  $\lsq$.  

\begin{remark}[Noise modelling] \label{rem:noise}
In the statistical context, LSQ minimization derives from maximum likelihood estimation using a Gaussian noise model on $\mathbf{\data}_{\Ho}$. From a statistical perspective, maximum likelihood estimation for Poisson noise on $\data$ might be more reasonable, resulting in $\mle(\signal) = \inner{\Fd( \signal), \boldsymbol{1}} - \inner{\boldsymbol{\log}(\Fd( \signal)), \data} \to \min_\signal$. As the focus in this paper is the derivation of an efficient reconstruction algorithm rather than statistical optimality, we work with $\lsq$. However, we expect that our strategy can also be applied to $\mle$ instead of $\lsq$.   
\end{remark}

\begin{remark}[Regularization] \label{rem:regularization} 
Due to the ill-conditioning of $\Hd( \signal) = \data_{\Hd} $, the reconstruction problem has to be regularized \cite{scherzer2009variational, benning2018modern}. In the context of LSQ minimization, a natural approach is  variational regularization, where one considers $ \norm{\Hd( \signal) - \data_{\Hd} }_2^2 / 2 + \alpha \mathcal{R}(\signal)$ instead  of $\lsq$ with a regularization functional $\mathcal{R}(\signal)$. Recently, in \cite{ebner2022plug}, the plug-and-play method has been identified as a regularization technique where the  regularization is incorporated  by a denoiser. Another  class  is given  by iterative regularization \cite{kaltenbacher2008iterative}, where regularization comes from early stopping. All these regularization methods require the gradient of $\lsq$, which we compute below.
\end{remark}

\section{Algorithm development}
\label{sec:algorithms}

In this section, we derive the proposed algorithms for MSCT based on channel preconditioning (CP). The first one (CP-full) integrates channel preconditioning into a gradient scheme. In the second algorithm (CP-fast), the derivative of the forward map is replaced by the derivative at zero. Both methods greatly reduce the number of iterations compared to standard gradient methods and the numerical cost per iteration compared to Newton type methods. Before presenting our algorithms, we start by computing derivatives and gradients and recall existing gradient  and Newton type methods.

\subsection{Derivatives computation}

Standard algorithms for minimizing \eqref{eq:lsq} require  derivative  of the  forward  map and the gradient of the LSQ functional  that we compute next. Recall the  original  and logarithmic MSCT forward operator $\Fd, \Hd \colon \R^{\Nx \times \Nc} \to  \R^{\Ny \times  \Nd}$  and the LSQ functional $\lsq$ defined by  \eqref{eq:fwd}, \eqref{eq:fwd-log} and \eqref{eq:lsq}.  We  equip   $\R^{\Nx \times \Nc}$ and $\R^{\Ny \times  \Nd}$ with the Hilbert space structure  induced by the standard inner product $\inner{\signalaux_1, \signalaux_2} = \sum_{i,j} \signalaux_1[i,j] \signalaux_2[i,j] $.   Further we use  $\Hd'[\signal] \colon \R^{\Nx \times \Nc} \to  \R^{\Ny \times  \Nd}$   to denote the derivative of $\Hd$ at  location $\signal \in \R^{\Nx \times \Nc} $ and $ \lsq' [\signal] \colon \R^{\Nx \times \Nc} \to \R$ and $\nabla \lsq [\signal] \in \R^{\Nx \times \Nc}$  to denote the derivative and the gradient  of  $\lsq$ at $\signal$, respectively.

\begin{remark}[Gradients,  inner products  and preconditioning]
By the definition of gradients, we have $\inner{ \nabla \lsq [\signal], \signalaux} = (\lsq'[\signal])^\ast(\signalaux)$, where $(\cdot)^\ast$ denotes the adjoint of a linear operator. Further, by the chain rule,  $\nabla \lsq [\signal] = (\Ho'[\signal])^ \cdot (\Ho (\signal) - \data_{\Hd} )$. Gradients and  adjoints depend on the chosen inner product. For example, the inner product $\inner{\signalaux_1, \Uo^{-1} \signalaux_2}$ on the image space with a positive-definite matrix $\Uo$ yields the modified gradient $\Uo \cdot \nabla \lsq [\signal]$. Choosing $\Uo$ such that $\Uo \cdot \nabla \lsq [\signal]$ has improved condition  significantly improves gradient based  methods for minimizing $\lsq$. \end{remark}

\begin{remark}[Some calculus rules] \label{rem:calculus}
For the following computation we use some  elementary calculus rules listed next.  Let $\Gd_1,\Gd_2 \colon \R^{\Nx} \to \R^{\Ny}$ be vector valued functions and  $\phi \colon \R \to \R$ a  scalar function with derivative $\psi \colon \R \to \R$. Then  for $\signal, \signalaux \in \R^{\Nx}$ we have   
\begin{align} \label{eq:calculus1}
	(\Gd_1  \odot \Gd_2)'[\signal] (\signalaux) 
	&= 
	 (\Gd_1'[\signal] (\signalaux)) \odot (\Gd_2(\signal)) 	+
	(\Gd_1(\signal))  \odot (\Gd'[\signal] (\signalaux))	
	\\ \label{eq:calculus2}
	\boldsymbol{\phi}_N' [\signal](\signalaux)
	&=
	\boldsymbol{\psi}_N(\signal) \odot \signalaux \,.
	\end{align}   
As usual we define the vector value functions $\boldsymbol{\phi}_N, \boldsymbol{\psi}_N  \colon \R^{\Nx} \to \R^{\Nx}$   by   pointwise application $\boldsymbol{\phi}_N(\signal) = (\phi(\signal_i))_i$   and $\boldsymbol{\psi}_N(\signal) = (\psi(\signal_i))_i$.
\end{remark}

We have  the following explicit expressions for derivatives, adjoints  and gradients.

\begin{theorem}[Derivatives computation] \label{thm:gradients}
Let  $\Fd, \Hd \colon \R^{\Nx \times \Nc} \to  \R^{B \times L}$  and  $\lsq$ be defined by~\eqref{eq:fwd}, \eqref{eq:fwd-log}, \eqref{eq:lsq}.  The the derivative of  $\Fd$, $\Hd$, the  adjoint and the gradient of  $\lsq$ are given  by 
\begin{align}   \label{eq:DF}
\Fd'[\signal] (\signalaux)^\intercal 
& =   -
 \So \cdot 
 \bigl( \virt_\signal \odot 
 (\Mo (\Ad \signalaux)^\intercal ) 
\\
 \label{eq:DH}
\Ho'[\signal] (\signalaux)^\intercal 
& = -\Fd(\signal)^{-T} \odot 
 (  
 \So \cdot 
 ( \virt_\signal \odot 
 (\Mo (\Ad \signalaux)^\intercal ) 
 )
 )
\\ \label{eq:DF-T}
\Fd'[\signal]^* (\dataaux) 
& 
=
-  \Ad^\intercal \cdot  (\Mo^\intercal  \cdot (\virt_\signal \odot (\So^\intercal  \dataaux^\intercal ))  )^\intercal
\\ \label{eq:DH-T}
\Ho'[\signal]^* (\dataaux)
& = -    \Ad^\intercal \cdot  (\Mo^\intercal  \cdot (\virt_\signal \odot (\So^\intercal (\Fd(\signal)^{-T} \odot \dataaux^\intercal) ))  )^\intercal
\\ \label{eq:nablaH}
\nabla \lsq [\signal] 
&= 
-  \Ad^\intercal \cdot   (\Mo^\intercal  \cdot (\virt_\signal \odot (\So^\intercal (\Fd(\signal)^{-T} \odot (\Hd (\signal) -  \data_{\Hd} )^\intercal) ))  )^\intercal
\,,
\end{align}
where $\virt_\signal \triangleq \boldsymbol{\exp}_{\Ne \times \Ny}(- \Mo (\Ad \signal)^\intercal)$ denote  virtual  spectrally resolved  data.  
\end{theorem}

\begin{proof}
By  \eqref{eq:fwd} and \eqref{eq:calculus2}  we have 
\begin{align*}
\Fd'[\signal]( \signalaux )^\intercal
 &=    \So \cdot ( \boldsymbol{\exp}_{\Ne \times \Ny}')[- \Mo (\Ad \signal)^\intercal](-\Mo (\Ad \signalaux)^\intercal ) 
 \\ 
 &=
- \So \cdot 
 ( \boldsymbol{\exp}_{\Ne \times \Ny}(- \Mo (\Ad \signal)^\intercal) \odot 
 (\Mo (\Ad \signalaux)^\intercal ) 
 )
 \\
 &=
 -
 \So \cdot 
 \bigl( \virt_\signal \odot 
 (-\Mo (\Ad \signalaux)^\intercal ) 
 \bigr)
  \,.
\end{align*}  
This  is \eqref{eq:DF}. Now, using  the   \eqref{eq:fwd-log},  \eqref{eq:calculus1}, \eqref{eq:calculus2}  we have 
\begin{align*}
\Ho'[\signal]( \signalaux )^\intercal 
 &= (\boldsymbol{\log}'_{B \times L}) [  \Fd(\signal)^\intercal  \oslash \Fd(0)^{T} ] \bigl(  \Fd'[\signal]( \signalaux )^\intercal \oslash \Fd(0)^{T} \bigr) 
\\
 &= 
    (  \Fd'[\signal]( \signalaux )^\intercal  \oslash \Fd(0)^{T} ) \oslash (  \Fd(\signal)^\intercal \oslash \Fd(0)^\intercal)
 \\ 
 &=
  \Fd'[\signal]( \signalaux )^\intercal  \oslash \Fd(\signal)^{T} 
 \\
 &=
 - \Fd(\signal)^{-T} \odot 
 \bigl(  
 \So \cdot 
 \bigl( \virt_\signal \odot 
 (\Mo (\Ad \signalaux)^\intercal ) 
 \bigr)
 \bigr)
  \,.
\end{align*}   
This  is \eqref{eq:DH}.  Next we turn over to the computation of the adjoints.  
We  this done only for   \eqref{eq:DH-T}  and    \eqref{eq:DF-T}  is verified  in a similar manner.   By elementary manipulations
\begin{align*}
 \inner{\Ho'[\signal]( \signalaux ) , \dataaux}
& = 
 \minner{ \Fd(\signal)^{-T} \odot 
 ( \So \cdot ( \virt_\signal \odot  (-\Mo \cdot  (\Ad \signalaux)^\intercal ) ) )
 ,  \dataaux^\intercal}
 \\
 &=
\inner{  
 (\Ad \signalaux)^\intercal 
 , - \Mo^\intercal (\virt_\signal \odot (\So^\intercal \cdot (\Fd(\signal)^{-T} \odot \dataaux^\intercal) ))}
 \\
 &=
\minner{  
 \Ad \signalaux 
 , -  (\Mo^\intercal  (\virt_\signal \odot (\So^\intercal \cdot (\Fd(\signal)^{-T} \odot \dataaux^\intercal) )))^\intercal }
 \\
 &=
\minner{  
 \signalaux 
  , -   \Ad^\intercal \cdot (\Mo^\intercal  \cdot (\virt_\signal \odot (\So^\intercal (\Fd(\signal)^{-T} \odot \dataaux^\intercal) ))  )^\intercal }  \,,
 \end{align*}
which   is \eqref{eq:DH-T}.  Finally,  with   $\nabla \lsq [\signal] = (\Ho'[\signal])^* (\Ho (\signal) -   \data_{\Hd} )$ and  \eqref{eq:DH-T} we obtain \eqref{eq:nablaH}.
\end{proof}

For CP-fast the derivative of $\Hd$ at zero plays a central role.

\begin{remark}[Derivative at zero] \label{rem:linearization}
Let us consider the derivative at the zero  image  $ \signal =0$.
In this case we have $\virt_0 := \boldsymbol{\exp}_{\Ne \times \Ny}(- \Mo (\Ad 0)^\intercal) =   \boldsymbol{1}_{\Ne \times \Ny}$   and therefore  $\Ho'[0] (\signalaux)^\intercal  = -     (\So \Mo  (\Ad \signalaux)^\intercal ) \oslash \Fd(0)^\intercal  $ and 
$\Ho'[0]^* (\eta)   = -    \Ad^\intercal   (\dataaux \oslash \Fd(0)) \cdot \So  \cdot  \Mo$.
Using that $\Fd(0)^\intercal = \So  \cdot  \boldsymbol{1}_{\Ne \times \Ny}$ 
we get   
 \begin{align} \label{eq:Lin0-B}
\Ho'[0] (\signalaux)^\intercal 
& =  - \Uo  (\Ad  \signalaux )^\intercal
\\ \label{eq:Lin0-ad-B}
\Ho'[0]^* (\dataaux) 
&  =  -   \Ad^\intercal \dataaux^\intercal  \Uo   
\\ \label{eq:U}
\Uo & \triangleq    (\So  \Mo) \oslash (\So  \cdot \boldsymbol{1}_{\Ne \times \Nc})     \,.
\end{align}
The derivative $\Ho'[0] (\signalaux)   =  -  \Ad  \signalaux \Uo^\intercal $ may also be seen as  the linearization of $\Ho$  around zero. It has been used previously in MSCT and can be   simply  derived by  first  order Taylor series approximation as we show   next.  In fact, with   $\Fd(0)_{\ell, b} =  \inner{ \So_b, \boldsymbol{1}_{\Ne \times 1}}$ we  et
 \begin{multline} \label{eq:taylor}
 \log  \left( \frac{  \inner{\So_b , \exp ( - \Mo (\Ad \signalaux)^\intercal_\ell ) }}{\inner{ \So_b, \boldsymbol{1}_{\Ne \times 1}}}  \right) 
\simeq 
 \log  \left( \frac{  \inner{ \So_b ,  (1 -  \Mo (\Ad \signalaux)^\intercal_\ell) }}{\inner{ \So_b, \boldsymbol{1}_{\Ne \times 1}} } \right)
\\ =
 \log  \left( 1- \frac{  \inner{ \So_b ,   \Mo(\Ad \signalaux)^\intercal_\ell  }}{\inner{ \So_b, \boldsymbol{1}_{\Ne \times 1}} } \right)
\simeq
 - \binner{  \frac{ \So_b}{ \inner{ \So_b, \boldsymbol{1}_{\Ne \times 1}}} ,    \Mo (\Ad \signalaux)^\intercal_\ell  } \,.
 \end{multline}
 The final expression   in \eqref{eq:taylor}  in matrix notation  is   \eqref{eq:Lin0-B}, \eqref{eq:U}.  
For  dual energy CT  (the case where $\Nc=\Nd=2$), the use  of the inverse of $\Uo$  has been  proposed in \cite{fessler2004method}. We emphasize that while we utilize the linearization $\Ho'[0] (\signalaux) = - \Ad \signalaux \Uo^\intercal$ as an auxiliary tool, we actually solve the full nonlinear problem. However, the linearized LSQ problem $\norm{\Ad \signalaux \Uo^\intercal - \data_\Ho }_2^2/2 \to \min_\signalaux$ is also of interest in its own. Theoretically proven convergent algorithms for such problems can be found in \cite{du2022convergence, rabanser2019analysis}.

  \end{remark}

The  derivatives  in Theorem \eqref{thm:gradients} have a clear composite structure that we will discuss next and exploit for our algorithms.

\begin{remark}[Composite structure of derivatives]
Consider $\signal  \in (\R^{\Nc})^{\Nx}$ as signal of size  $\Nx$ with  $\Nc$ channels (each channel  is a material) and  the data $\data_\Ho  \in (\R^{\Nd})^{\Ny}$ as signal of size  $\Ny$ with $\Nd$ channels (each channel  is an energy bin). 
Then we can write $\Hd (\signal) =  \boldsymbol{\Phi}(\Ad \signal) $ where the nonlinear function 
 \begin{equation}  \label{eq:phi}
  \boldsymbol{\Phi} (Z) =  \boldsymbol{\log}_{\Ny \times  \Nd} (   \boldsymbol{\exp}_{L\times E} ( -  Z \Mo^\intercal  ) \cdot \underline{\So}^\intercal   )
  \end{equation}
operates on the multichannel  sinogram $\Ad \signal$   along  the  (horizontal) channel dimension and $\underline \So = \So  \oslash ( \So  \cdot  \boldsymbol{1}_{\Ne \times \Ny}  )$ are the normalized effective spectra.  Application of the chain rule  and some  computations results in  
 \begin{align}  \label{eq:Dphi1}
 \Hd' [\signal](\signalaux) &=  \boldsymbol{\Phi}'[\Ad \signal] (\Ad \signalaux)
 \\ \label{eq:Dphi2}
 \boldsymbol{\Phi}'[Z](\zeta)
  &=   - (   \boldsymbol{\exp}_{L\times E} ( -  Z \Mo^\intercal  )  \odot 
 ( \zeta \cdot \Mo^\intercal ) ) \cdot  \So^\intercal ) \oslash  ( \boldsymbol{\exp}_{L\times E}( - Z \Mo^\intercal ) \So^\intercal) 
  \\ \label{eq:Dphi3}
 \boldsymbol{\Phi}'[Z]^*(\dataaux)
  &=  
   - (  \boldsymbol{\exp}_{L\times E} ( -  Z \Mo^\intercal  )  \odot  ((\dataaux  \oslash  ( \boldsymbol{\exp}_{L\times E}( - Z \Mo^\intercal ) \So^\intercal))   \cdot \So) ) \cdot \Mo \,.
\end{align}
from which we recover  \eqref{eq:DH}, \eqref{eq:DH-T}. Further, for the zero material sinogram $Z=0$ we get  $\boldsymbol{\Phi}'[0](\zeta)
  =   -    \zeta \Uo^\intercal $ and $\boldsymbol{\Phi}'[0]^*(\dataaux)
  =   -    \dataaux \Uo $ with $\Uo =    (\So \cdot \Mo) \oslash (\So \cdot \boldsymbol{1}_{\Ny \times \Nc})  $ as in Remark~\ref{rem:linearization}.
  Equations \eqref{eq:Dphi1}, \eqref{eq:Dphi2}, \eqref{eq:Dphi3}  factorise the derivative $\boldsymbol{\Phi}'[Z]$ and its adjoint into  two separate parts:   A high dimensional ill-posed but  linear  part $\Ad$ operating in the pixel dimension and  small size  well-posed but nonlinear  part operating in the channel dimension.  Our  algorithms will  target this  structure for fast and effective iterative updates.
 \end{remark}

\subsection{Gradient and Newton-type one-step algorithms}

It is helpful to start  with gradient based method for minimizing the LSQ functional \eqref{eq:lsq}. Our first method CP-full can be seen as a modified version of a hybrid between the standard gradient iteration (or Landweber's method) and the Gauss-Newton method, so we describe these methods. Our second  method CP-fast  involves a simplification based on linearization around zero.

Gradient based one-step algorithms for solving the MSCT problem $\Ho \signal = \data_\Ho$ using the LSQ functional   $\lsq (\signal) $ start with the optimality condition  $\nabla \lsq[ \signal] =  \Ho'[\signal]^*    
( \Hd (\signal) - \data_{\Hd} ) = 0$ and fixed point equations derived from it.   Applying a non-stationary positive-definite  preconditioner   $ \Qo_k$  and a step size  $\stepsize_k >0$ results in   
\begin{equation} \label{eq:grad-iter}
\signal^{k+1} = \signal^k - \stepsize_k \cdot \Qo_k  \Ho'[\signal^k]^*    
\bigl( \Hd (\signal^k) - \data_{\Hd} \bigr)  \,.
\end{equation}
Explicit expressions for  $\Ho$  and   $\Ho'[\signal^k]^*$ are given by \eqref{eq:fwd-log} and  \eqref{eq:DH-T}.
Particular choices for   the preconditioner and the step size   yield various iterative solution  methods including Landwebers iteration, the steepest descent method,  Gauss-Newton iteration, Newton-CG iterations,   or Quasi-Newton methods. To motivate our algorithm it is most  educational  to  discuss the  Landweber and the Gauss-Newton iteration.         

\begin{description}
\item[Landwebers method:]   
In the context of inverse problems, the standard gradient method with a constant step size $\stepsize$ is known as the (nonlinear) Landweber iteration $\signal^{k+1} = \signal^k - \stepsize \cdot  \Ho'[\signal^k]^*    
( \Hd (\signal^k) - \data_{\Hd} )$  which  is   \eqref{eq:grad-iter} for the case where $ \Qo_k$ is the identity and $\stepsize_k= \stepsize$.  
Landweber's iteration is stable, robust, and easy to implement. It is even applicable in ill-posed cases where, with an appropriate stopping criterion, it serves as a regularization method \cite{hanke1995convergence}. On the other hand, it is also known to be slow in the sense that many iterative steps are required. In our case, this is due to the ill-conditioning of the forward operator.

\item[Gauss-Newton method:]   
Several potential accelerations of Landweber's method exist, and preconditioning seems one of the most natural ones. In the context of nonlinear least squares, the Gauss-Newton method and its variants are well-established and effective. In this case one chooses the preconditioner $\Qo_k = (\Ho'[\signal^k]^\ast  \Ho'[\signal^k])^{-1}$  in \eqref{eq:grad-iter} which  results  in 
\begin{equation} \label{eq:gauss-newton}
\signal^{k+1} = \signal^k - \stepsize_k \cdot  (\Ho'[\signal^k]^\ast \Ho'[\signal^k])^{-1}   \Ho'[\signal^k]^*    
\bigl( \Hd (\signal^k) - \data_{\Hd} \bigr) \,.
\end{equation}
While significantly  reducing the required number  of iterations, the  Gauss-Newton iteration \eqref{eq:gauss-newton}, however, is numerically costly as it requires  inversion of the non-stationary normal operator $\Ho'[\signal^k]^\ast \Ho'[\signal^k]$ in each iterative update. Moreover, due to ill-conditioning, the inversion needs to be regularized \cite{hanke1997regularizing,rieder1999regularization,kaltenbacher2008iterative}.  The algorithms proposed in this paper use simplifications that do not need to be regularized and avoid the costly inversion of the normal operator.
\end{description}

\subsection{Proposed algorithms} \label{sec:alg}

Now we move on to the proposed  iterative algorithms for MSCT. We start by  with CP-full which is  gradient-based  algorithm  with channel  preconditioning. We then  derive CP-fast   which is  derivative-free iterative algorithm using  a stationary adjoint problem.

\begin{description}
\item[CP-full:]
The first proposed algorithm is an instance of \eqref{eq:grad-iter}. Instead of no preconditioning as in Landweber's method or the costly preconditioning $\Qo_k = (\Ho'[\signal^k]^\ast \Ho'[\signal^k])^{-1}$ as in the Gauss-Newton method, we propose preconditioning with the channel mixing term  $\boldsymbol{\Phi}$ only. That is, we exploit the factorization $\Ho(\signal) = \boldsymbol{\Phi}(\Ad \signal)$ and propose  the choice   $ \Qo_k = ( \boldsymbol{\Phi}'[\Ad \signal^k]^\ast    \boldsymbol{\Phi}'[\Ad \signal^k] )^{-1}$  for the preconditioner. This result in the CP-full iteration  
\begin{equation} \label{eq:cp-full}
	\signal^{k+1} = 
	\signal^k - \stepsize_k \cdot  \Ad^\intercal \cdot (  \boldsymbol{\Phi}'[\Ad \signal^k]^\ast   ( \boldsymbol{\Phi}'[\Ad \signal^k] )^{-1} \boldsymbol{\Phi}'[\Ad \signal^k]^\ast 
(  \Hd  (\signal^k ) - \data_{\Hd} ) \,.
\end{equation}
While efficiently addressing the nonlinearity via a Gauss-Newton-type preconditioner in the channel dimension, it is computationally much less costly than the full Gauss-Newton update. 
Instead of inverting $\Ho'[\signal^k]^\ast \Ho'[\signal^k]$, which in matrix form has size $(\Nx \Nc) \times (\Nx \Nc)$ in the Gauss-Newton method, it requires inversion of the smaller $M \times M$ matrices $\boldsymbol{\Phi}'[\Ad \signal^k]^\ast \boldsymbol{\Phi}'[\Ad \signal^k]$ only, which can be done separately for each pixel in the projection domain. Assuming $\Nc,\Nd = \mathcal O(1)$ and $\Ny = \mathcal O(\Nx)$, this dramatically reduces the cost of preconditioning from $\mathcal O(\Nx^3)$ to $\mathcal O(\Nx)$ per iterative update.

\item
\item[CP-fast:]
In the derivative-free version, we go one step further and completely avoid the derivative $\Hd'[\signal^k]$. For that  purpose we replace  the derivative  $\boldsymbol{\Phi}'[\Ad \signal^k]$ in \eqref{eq:cp-full} by the derivative at zero. According to Remark~\ref{rem:linearization} we have  $\boldsymbol{\Phi}'[0](\zeta) =   -    \zeta \Uo^\intercal $   with $\Uo =    (\So \cdot \Mo) \oslash (\So \cdot \boldsymbol{1}_{\Ne \times \Nc}) $. Now, with $\Uo^\ddagger = (\Uo^\intercal \Uo)^{-1} \Uo^\intercal$ denoting the pseudoinverse of $\Uo$, we arrive at the iterative update
\begin{equation} \ \label{eq:cp-fast}
\signal^{k+1} = \signal^k - \stepsize_k \cdot  \Ad^\intercal \cdot    
(  \Hd  (\signal^k ) - \data_{\Hd} ) \cdot (\Uo^\ddagger)^\intercal 
\,.
\end{equation}
We refer to \eqref{eq:cp-fast} as derivative-free fast channel-preconditioned  (cp-fast) iteration. It only  involves the derivative at zero, which can be computed once before the actual iteration. In this sense, it is actually derivative-free and  fast.  
It can be interpreted as using the full nonlinear model for the forward problem, the linearization at zero for the adjoint problem, and including channel preconditioning.
\end{description}

Both iterations \eqref{eq:cp-full} and \eqref{eq:cp-fast} are of fixed-point type and we therefore expect convergence for sufficiently small step sizes.   Theoretically  proving  convergence seems possible  but is beyond the scope of this paper. As  \eqref{eq:cp-full}  is  of gradient  type it seems easier to derive  convergence for CP-full while for the derivative free  version CP-fast such a proof seems challenging. Note further that  for  the results presented below we integrated a positivity constraint by alternating  iterative updates with  the orthogonal projection onto the  cone of non-negative images.

\section{Numerical simulations}

We compared our algorithms CP-full and CP-fast to existing iterative one-step algorithms in MSCT.   Our evaluation builds on  \cite{Mory_2018}, which compares five such algorithms and provides open source code (\url{https://github.com/SimonRit/OneStepSpectralCT}) that is used for our results. We compare CP-full and CP-fast  with the best performing one of \cite{Mory_2018}, and further with a two-step method.

\begin{figure}
\includegraphics[width=0.48\textwidth]{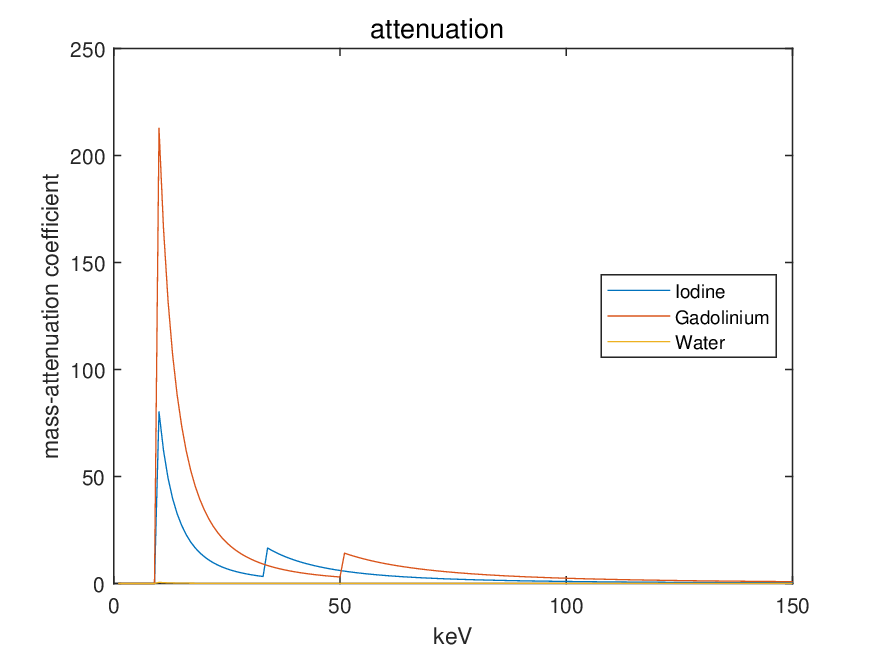}
\includegraphics[width=0.48\textwidth]{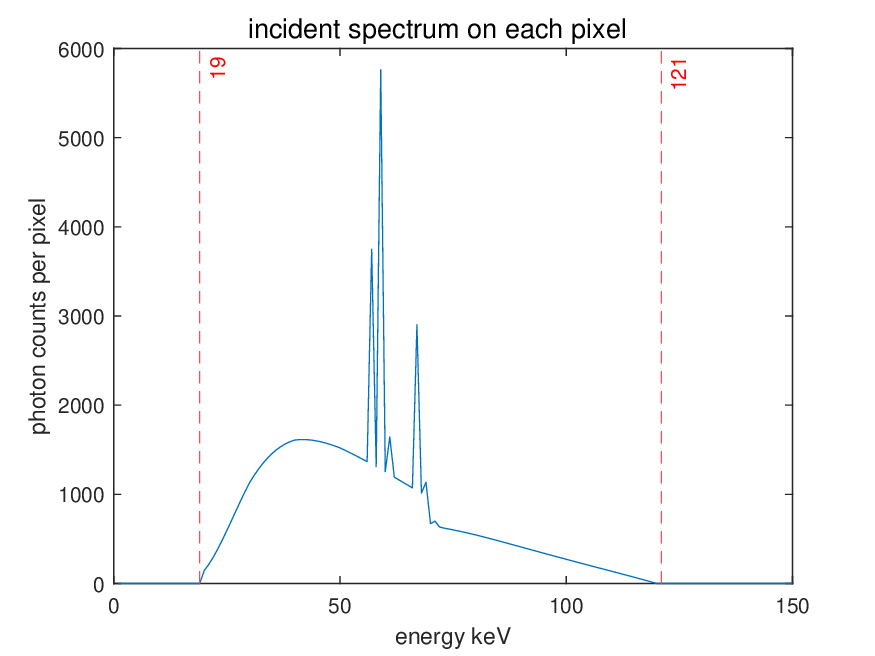} \\[0.2em]
\includegraphics[width=0.48\textwidth]{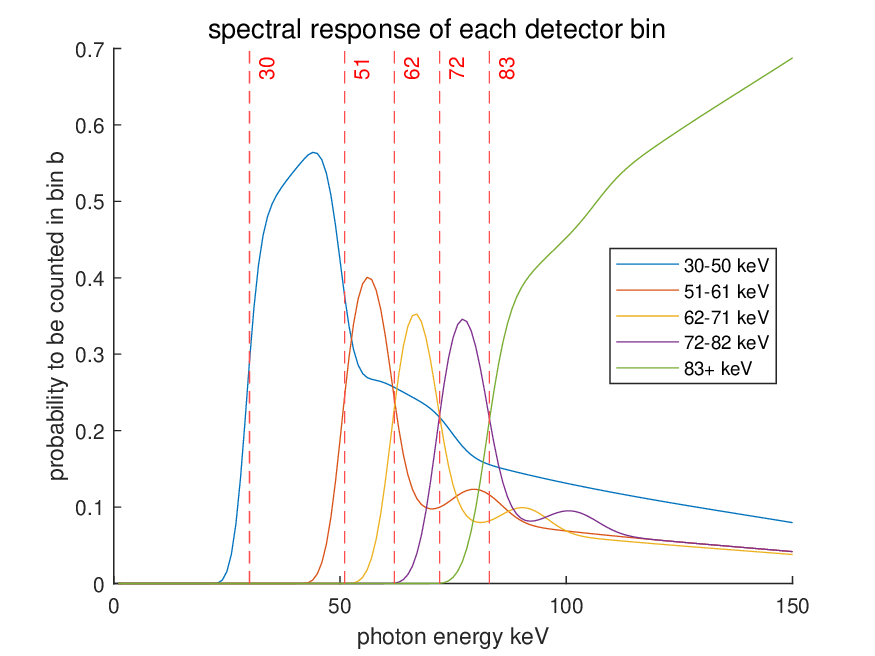}
\includegraphics[width=0.48\textwidth]{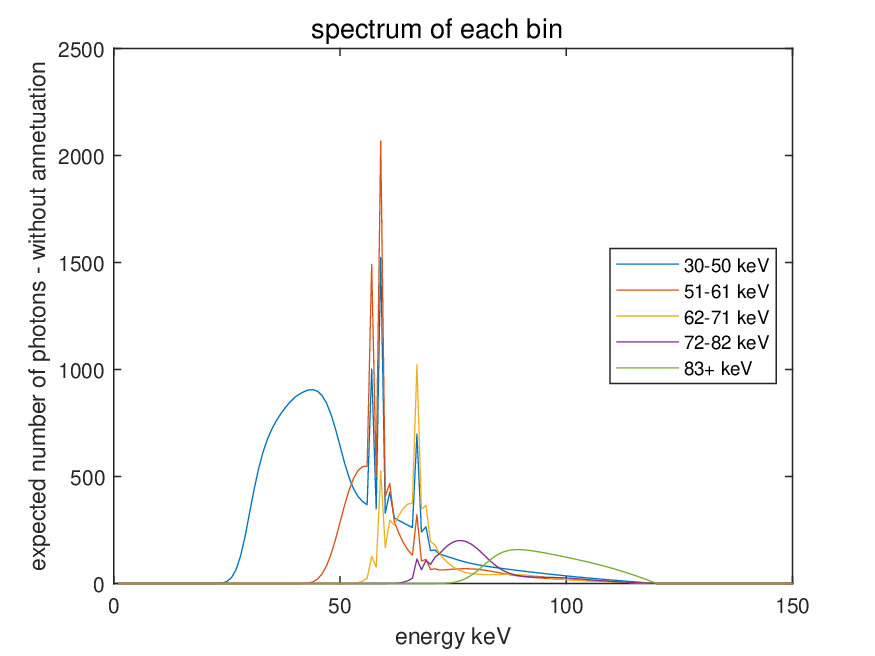}
\caption{Physical  parameters determining the  forward model. Top Left: Attenuation functions. Top right: Incident spectrum. Bottom left: Spectral response of the detectors. Bottom right: effective spectra.}
\label{fig:para}
\end{figure}

\subsection{Comparison methods}

The work \cite{Mory_2018} compares the following iterative one-step algorithms for MSCT in terms of memory usage and convergence speed to reach a fixed image quality threshold:
\begin{itemize}
\item \cite{cai2013full} derives a non-linear CG method for a weighted LSQ term.
\item \cite{long2014multi,weidinger2016polychromatic} derive surrogate approaches for Poisson maximum likelihood.
\item \cite{mechlem2017joint} extends \cite{weidinger2016polychromatic} by including Nesterov's momentum acceleration.
\item \cite{barber2016algorithm} generalizes the Chambolle–Pock algorithm \cite{chambolle2011first} to non-convex functionals.
\end{itemize}
Specifically, \cite{Mory_2018} found the algorithm of \cite{mechlem2017joint} (referred to as Mechlem2018) to be significantly faster than the other four methods and thus we use it for comparison.

In addition, we compare with the algorithm \cite{niu2014iterative} (referred to as Niu2014) as a prime example of an image domain two-step method. They use a penalised weighted least squares estimation technique applied to an empirically  linear model.  Note that more recently, data-driven methods based on neural networks and deep learning have also been proposed. Such methods are beyond the scope of this manuscript and we refer the interested reader to the  review articles \cite{bousse2023systematic,arridge2021overview}.

\begin{figure}[htb!]
\includegraphics[width=0.98\textwidth]{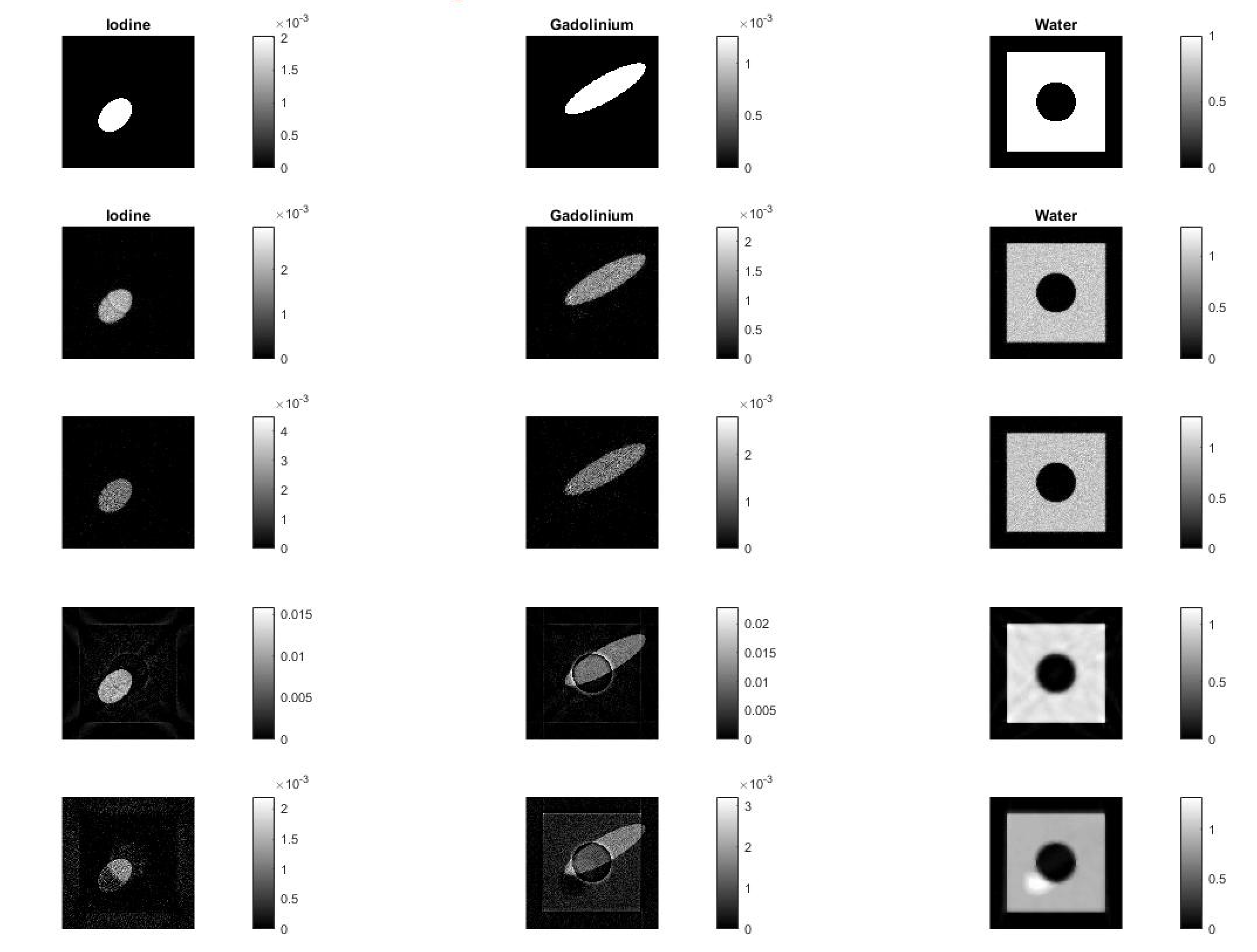}
\caption{Ground truth phantom (top row), and reconstructions using the CP-fast (second row), the CP-full (third row), Mechlem2018 (fourth row), and the two-step algorithm (bottom).}
\label{fig:recon}
\end{figure}

\subsection{Numerical implementation}

For the presented results we build on the Matlab code of \cite{Mory_2018}, which we extend with our algorithms. In particular, we work with $\Nc=3$ base materials (water, iodine, and gadolinium) and $\Nd=5$ energy bins. The energy variable is discretized using $\Ne = 150$ uniform nodes between 0 and 150 keV. The attenuation functions and energy spectra used are shown in Figure \ref{fig:para}.  We use $\Nx = 256 \times 256$ image pixels and $\Ny = 262450$ line integrals for the Radon transform. In particular the code \url{https://github.com/SimonRit/OneStepSpectralCT} creates matrices   
\begin{itemize}
\item $\Mo \in \R^{150 \times 3}$ for the base materials; 
\item  $\So \in \R^{5 \times 150} $ for the effective energy spectra; 
\item $\Ad \in \R^{L \times N^2}$ for the   Radon transform.      
\end{itemize}
After row normalizing  $\underline \So = \So  \oslash ( \So  \cdot  \boldsymbol{1}_{\Ne \times \Ny}  )$ we have  $\Ho (\signal) = \boldsymbol{\log} ( \boldsymbol{\exp}( - \Ad \cdot \signal \cdot \Mo^\intercal ) \cdot \underline{\So}^\intercal  )$ for the  MSCT forward model. Further, noisy data $\data$ are created with a different realistic forward model and Poisson noise added.

Besides  $\Mo$, $\So$,  $\Ad$,  $\Ho$ and $\data$ we require implementations of  $\Uo^\ddagger = (\Uo^* \Uo)^{-1} \Uo^*$  for CP-fast  and $\Uo_k^\ddagger = (\boldsymbol{\Phi}'[\Ad \signal^k]^\ast  \boldsymbol{\Phi}'[\Ad \signal^k] )^{-1} \boldsymbol{\Phi}'[\Ad \signal^k]^\ast$ for  CP-full.  Computing $\Uo^\ddagger$ is trivial and can be done in advance;  $\Uo_k^\ddagger$  are computed in each  step of CP-fast using \eqref{eq:Dphi2}, \eqref{eq:Dphi3}.

\subsection{Results}

Reconstruction results using the proposed algorithms CP-fast (top row) and CP-full (second row) and the two comparison methods Mechlem2018  (row three) and Niu2014 (bottom row) can be seen in Figure~\ref{fig:recon}. The phantom shown in the top row is made out of iodine (left), gadolinium (middle), and water (right).  Note that in all cases, we use noisy data and iterations with minimal $\ell^2$-reconstruction error $\norm{\signal_m^k - \signal_m^\star}_2/ \norm{\signal_m^\star}2$ where $\signal_m^\star$ is the ground truth. Note that the iterations are all performed on the same standard laptop, where one iteration of CP-full takes around 6 seconds, one iteration of CP-fast takes around one second  one iteration of Mechlem2018  about 4 seconds. Figure~\ref{fig:err} shows the evolution of the relative $\ell^2$-reconstruction error for various one-step methods. Note that for CP-fast, the minimum error in Iodine and Gadolinium is reached at approximately the same number of iterations, which shows efficient preconditioning and is important in application. Furthermore, note that we do not enforce that the sum over the three density images is one, but in the example, it indeed does not hold. The proposed algorithms turned out to be more stable than Mechlem2018 and produce better results. In particular, CP-full gives the best results while CP-fast is fastest.

\begin{figure}[htb!]
\includegraphics[width=0.98\textwidth]{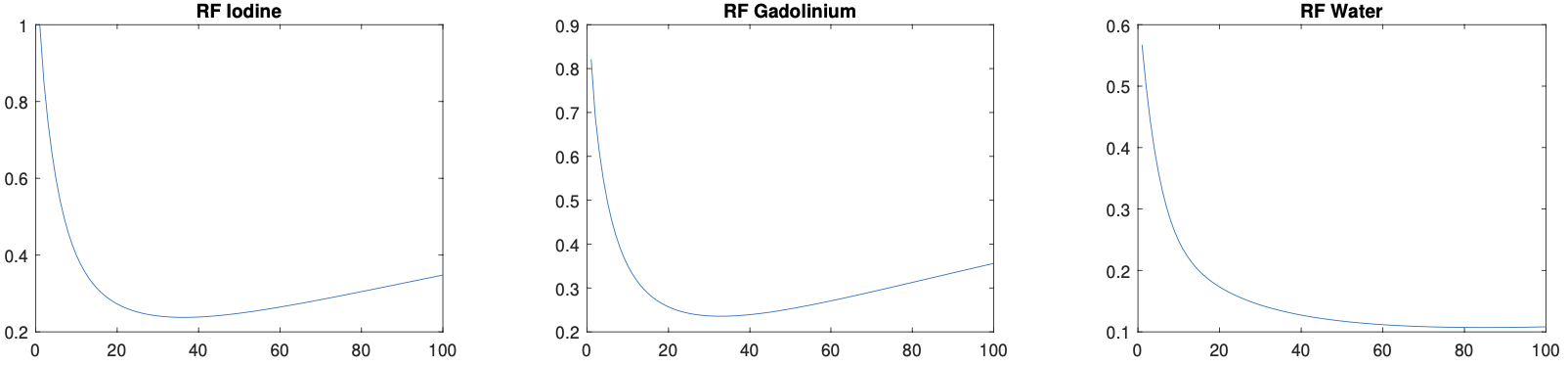} \\[0.2em]
\includegraphics[width=0.98\textwidth]{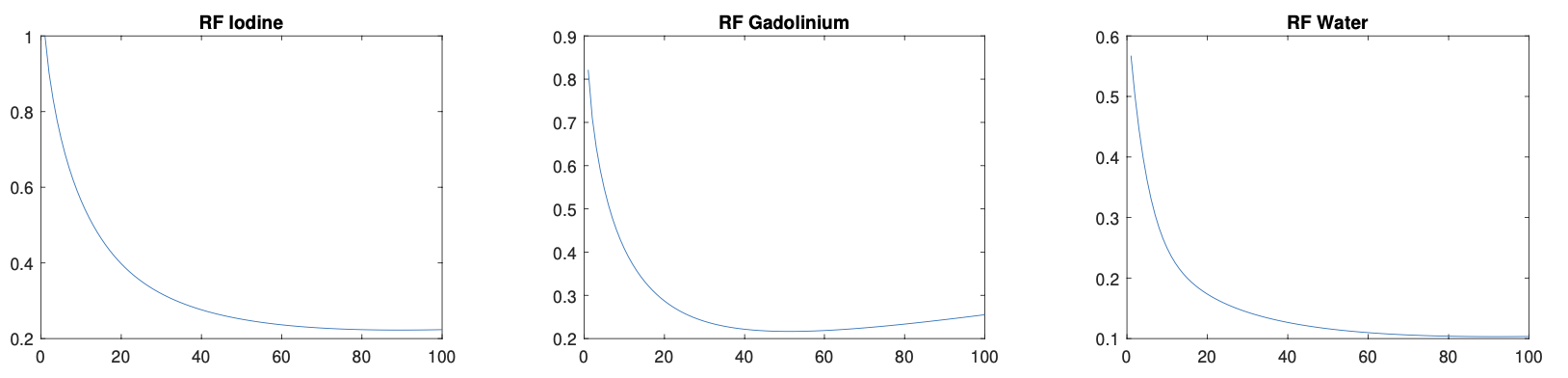} \\[0.2em]
\includegraphics[width=0.98\textwidth]{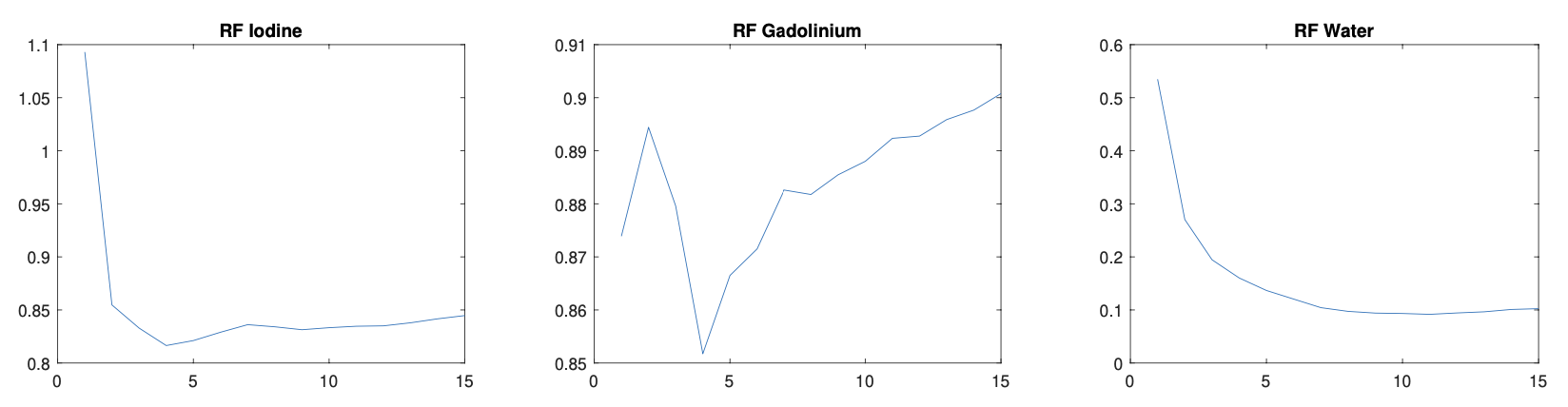} 
\caption{Relative reconstruction  error  using proposed CP-fast (top), proposed CP-full (middle)  and Mechlem2018 (bootom)  as a function of the  iteration index. }
\label{fig:err}
\end{figure}

\section{Conclusion and outlook}

Image reconstruction in MSCT requires the solution of a nonlinear ill-posed problem. Iterative one-step methods are known to be accurate for this purpose. In this work, we propose two generic algorithms named CP-full (channel-preconditioned full gradient iteration) and CP-fast (channel-preconditioned fast iteration). Both algorithms use preconditioning in the channel dimension only, which considerably accelerates the updates compared to full preconditioning used by Newton-type methods. CP-fast replaces the derivative in the channel non-linearity with linearization at zero, making it even more efficient. Both algorithms turn out to be fast and robust.

There are several future directions emerging from our work. First, proving the convergence of the two algorithms and demonstrating their regularization properties is important. Second, we will combine them with more realistic noise priors such as Poisson noise, resulting in the maximum likelihood estimation (MLE) functional. Additionally, we will integrate explicit image priors, use plug-and-play strategies, and incorporate learned components.

\end{document}